\documentclass[letterpaper,11pt,reqno]{amsart}

\makeatletter
\usepackage{amssymb}
\usepackage{latexsym}
\usepackage{amsbsy}
\usepackage{amsfonts}
\usepackage{graphicx}
\usepackage{enumerate}
\usepackage{enumitem}
\usepackage{mathtools}
\usepackage{subcaption}
\usepackage{color}
\usepackage{tcolorbox}
\usepackage{url}
\usepackage{comment}
\usepackage{appendix}

\usepackage{tikz}

\def\marginpar#1{\ignorespaces}

  \newcommand{\beq}{\begin{equation}}
    \newcommand{\eeq}{\end{equation}}
    
    \newcommand{\bal}{\begin{align}}
    \newcommand{\eal}{\end{align}}
    \newcommand{\bals}{\begin{align*}}
    \newcommand{\eals}{\end{align*}}
    
\newtheorem{theorem}{Theorem}[section]

\newtheorem{lemma}[theorem]{Lemma}
\newtheorem{proposition}[theorem]{Proposition}
\newtheorem{corollary}[theorem]{Corollary}

\numberwithin{equation}{section}
\makeatother
\begin{document}
\title[Pitman and  the Gorin-Shkolnikov identity]{From Pitman's local times representation to the Gorin-Shkolnikov identity and beyond}

\author[W. Tang]{Wenpin Tang}

\address[W. Tang]{Department of Industrial Engineering and Operations Research, Columbia University, S.W. Mudd Building, 
500 W 120th St, New York, NY 10027} 
\email{wt2319@columbia.edu}

\date{\today}

\begin{abstract}
We give a novel proof of the Gorin-Shkolnikov identity based on Pitman's SDE representation of Brownian excursion local times.
More generally, we derive a family of Gaussian identities for nonlinear functionals of excursion local times, which includes Hariya's identity.
These results extend naturally to reflected Brownian bridges conditioned on their local times  and to Brownian meanders.
\end{abstract}

\maketitle



$$~$$
\quad The purpose of this note is to provide a novel proof of the following theorem,
which first appeared as a by-product in the study of the stochastic Airy operator \cite{GS18}. 

\begin{theorem}[\cite{GS18}, Corollary 2.3]
\label{thm:main}
Let $\{e(t)\}_{0 \le t \le 1}$ be a Brownian excursion on $[0,1]$
and $\{l_x\}_{x \ge 0}$ be local times process\footnote{Here $\ell_y$ is the local times of $\{e(t)\}_{0 \le t \le 1}$ at level $y \ge 0$. See e.g., \cite[Chapter VI]{RY} for background.}.
Then:
\begin{equation}
\label{eq:GSid}
\int_0^1 e(t)dt - \frac{1}{2} \int_0^\infty \ell_x^2 dx \stackrel{d}{=} \mathcal{N}\left(0,\frac{1}{12}\right),
\end{equation}
where $\mathcal{N}(m,\sigma^2)$ denotes a Gaussian random variable with mean $m$ and variance $\sigma^2$.
\end{theorem}

\quad The fact that the left side of \eqref{eq:GSid} has zero mean may not be too surprising because it is known \cite{BY87, Jeu} that 
\begin{equation*}
\int_0^1 e(t)dt \stackrel{d}{=} \frac{1}{2} \int_0^\infty \ell_x^2 dx.
\end{equation*}
See \cite{Jan07} for the discussions on the distribution of the Brownian area $\int_0^1 e(t)dt$.
What's interesting is that the difference of these two terms is Gaussian with variance $\frac{1}{12}$.
The work \cite{Hai16} gave a pathwise proof of \eqref{eq:GSid} via Jeulin's identity \cite{Jeu},
and \cite{Cla21} provided a random forest interpretation of this identity.

\quad Here we give yet another proof of the identity \eqref{eq:GSid},
which relies on Pitman's SDE representation of $\{\ell_x\}_{x \ge 0}$ \cite{Pit99}.

\quad We start with the following lemma, which expresses the left side of \eqref{eq:GSid} 
in terms of $\ell_x$ and $I_x:=1 - \int_0^x \ell_y dy$.
\begin{lemma} \label{lem:lI}
We have:
\begin{equation}
\label{eq:path}
\int_0^1 e(t)dt -\frac{1}{2} \int_0^\infty \ell_x^2 dx = \int_0^\infty \left(I_x - \frac{1}{2} \ell_x^2 \right) dx.
\end{equation}
\end{lemma}
\begin{proof}
By the occupation time formula, we get:
\begin{equation*}
\begin{aligned}
\int_0^1 e(t)dt  &= \int_0^\infty x \ell_x dx \\
& = \int_0^\infty \int_x^\infty \ell_y dy\, dx = \int_0^\infty I_x dx,
\end{aligned}
\end{equation*}
where the second equality is by Fubini's theorem.
\end{proof}

\quad The following Pitman's theorem gives an SDE representation of $\{\ell_x\}_{x \ge 0}$.
\begin{lemma}[\cite{Pit99}, Corollary 5]
The process $\{\ell_x\}_{x \ge 0}$ solves the SDE:
\begin{equation}
\label{eq:Pitman}
d\ell_x = \left( 4 - \frac{\ell_x^2}{1 - \int_0^x \ell_y dy}\right) dx + 2 \sqrt{\ell_x} dB_x, \quad \ell_0 = 0,
\end{equation}
stopped at $0$.
Here $\{B_x\}_{x \ge 0}$ is standard Brownian motion.
\end{lemma}

\quad Now we are ready to prove Theorem \ref{thm:main}.
\begin{proof}[Proof of Theorem \ref{thm:main}]
Applying It\^{o}'s formula to $\ell_x I_x$, we get:
\begin{equation*}
\begin{aligned}
d(\ell_x I_x) &= \ell_x dI_x + I_x d\ell_x \\
& = (4 I_x - 2 \ell_x^2)dx + 2 I_x \sqrt{\ell_x} dB_x.
\end{aligned}
\end{equation*}
where we use that $dI_x = -\ell_x dx$ and the SDE \eqref{eq:Pitman}.
So $(I_x - \frac{1}{2} \ell_x^2)dx = \frac{1}{4} d(\ell_x I_x) - \frac{1}{2}I_x \sqrt{\ell_x} dB_x$.
Integrating over $[0,\infty)$, we have:
\begin{equation}
\label{eq:4}
\int_0^\infty \left(I_x - \frac{1}{2} \ell_x^2\right)dx = -\frac{1}{2} \int_0^\infty I_x \sqrt{\ell_x}dB_x,
\end{equation}
where we use that $\ell_x I_x \to 0$ as $x \to \infty$.
Note that the right side of \eqref{eq:4} is a Gaussian variable with mean $0$,
and variance
\begin{equation}
\label{eq:5}
\frac{1}{4} \int_0^\infty I_x^2 \ell_x dx = \frac{1}{4} \int_0^1 z^2 dz = \frac{1}{12},
\end{equation}
because $dI_x = -\ell_x dx$, and $I_0 = 1$ and $I_\infty = 0$.
Combining \eqref{eq:path}, \eqref{eq:4} and \eqref{eq:5} yields \eqref{eq:GSid}.
\end{proof}

\quad Next, we provide several generalizations encompassing Theorem \ref{thm:main} and Pitman's SDE representation of local times. 

\smallskip
{\bf (A) General test function}. The following proposition is a simple generalization of \eqref{eq:4}. 
\begin{proposition}
\label{coro:test}
Assume that $\phi: \mathbb{R} \to \mathbb{R}$ be sufficiently smooth and $\int_0^1\phi(x)^2dx < \infty$.
We have:
\begin{equation}
\label{eq:6}
\int_0^\infty \phi(I_x) dx - \frac{1}{4} \int_0^\infty \ell_x^2 \left( \frac{\phi(I_x)}{I_x} + \phi'(I_x) \right)dx 
\stackrel{d}{=} \mathcal{N}\left(0, \frac{1}{4} \int_0^1\phi(x)^2 dx\right).
\end{equation}
\end{proposition}
\begin{proof}
It suffices to apply It\^{o}'s formula to $\ell_x \phi(I_x)$, and the rest of the proof is the same as above.
\end{proof}

\quad In fact, the term on the left side of \eqref{eq:6} can be viewed as an analogue of Wick-renormalized local times,
whose remainder is Gaussian.
Taking $\phi(x) = x$, the identity \eqref{eq:6} yields Theorem \ref{thm:main}.
Specializing to $\phi(x) = x^n$, $n \ge 1$ gives the following result of Hariya \cite{Hai16}.
\begin{corollary}[\cite{Hai16}, Proposition 2.2]
\label{coro:gen}
For $0 \le t_1 < \cdots < t_n \le 1$, we have:
\begin{equation}
2 \int_{[0,1]^n} \min(e(t_1), \ldots, e(t_n))dt_1 \cdots dt_n - \frac{n+1}{2} \int_0^\infty  I_x^{n-1} \ell_x^2dx
\stackrel{d}{=} \mathcal{N}\left(0, \frac{1}{2n+1} \right),
\end{equation}
\end{corollary}

\smallskip
{\bf (B) Level $a \ge 0$}. The following proposition generalizes \eqref{eq:6} to integrals from $a \ge 0$.
\begin{proposition}
Let $\{\mathcal{F}_a\}_{a \ge 0}$ be the excursion filtrations\footnote{See \cite{Mc86, Wil79} for background on excursion filtrations.}.
Assume that $\phi: \mathbb{R} \to \mathbb{R}$ be sufficiently smooth and $\int_0^1\phi(x)^2dx < \infty$.
We have:
\begin{equation}
\int_a^\infty \phi(I_x) dx - \frac{1}{4} \int_a^\infty \ell_x^2 \left( \frac{\phi(I_x)}{I_x} + \phi'(I_x) \right)dx \,\bigg|\, \mathcal{F}_a
\stackrel{d}{=} \mathcal{N}\left(-\frac{\ell_a \phi(I_x)}{4}, \frac{1}{4} \int_{I_x}^\infty \phi(y)^2dy \right).
\end{equation}
\end{proposition}

\quad In particular, taking $\phi(x) = x$ yields the following corollary.
\begin{corollary}
We have on $\{I_a > 0\}$, 
\begin{equation}
\frac{1}{I_a^{3/2}}\left(\int_0^1 (e(t)-a)_{+} - \frac{1}{2} \int_a^\infty \ell_x^2 dx + \frac{1}{4} \ell_a I_a\right) 
\stackrel{d}{=} \mathcal{N}\left(0, \frac{1}{12} \right),
\end{equation}
which is independent of $\mathcal{F}_a$.
\end{corollary}

\smallskip
{\bf (C) Reflected Brownian bridge}.
A Brownian excursion spends ``zero" time at level $0$.
It is easy to extend Corollary \ref{coro:gen} to reflected Brownian bridges conditioned on its local times (see \cite{Aldous98} for background).
\begin{proposition} \label{prop:RB}
For $c \ge 0$, let $\{\ell^c_x\}_{x \ge 0}$ be the local times process of a reflected Brownian bridge $\{|B^{\tiny \mbox{br}}_t|\}_{0 \le t \le 1}$
conditioned on $\ell_0 = c$,
and $I_x^c: = 1 - \int_0^x \ell^c_y dy$.
Assume that $\phi: \mathbb{R} \to \mathbb{R}$ be sufficiently smooth and $\int_0^1\phi(x)^2dx < \infty$.
We have:
\begin{equation}
\label{eq:7}
\int_0^\infty \phi(I_x^c) dx - \frac{1}{4} \int_0^\infty (\ell_x^c)^2 \left( \frac{\phi(I_x^c)}{I_x^c} + \phi'(I_x^c) \right)dx 
\stackrel{d}{=} \mathcal{N}\left(-\frac{c \phi(1)}{4}, \frac{1}{4} \int_0^1\phi(x)^2 dx\right).
\end{equation}
\end{proposition}

\quad The proof is the same as Proposition \ref{coro:test} by applying \cite[Theorem 4]{Pit99} that $\{\ell_x^c\}_{x \ge 0}$ solves the SDE \eqref{eq:Pitman} with the initial condition $\ell_0^a = c$.
The  conditioning $\ell_0 = c$ yields a nonzero mean $-\frac{c \phi(1)}{4}$.

\smallskip
{\bf (D) Brownian meander}.
We conclude with a version of Theorem \ref{thm:main} for Brownian meander (see \cite[Chapter 0]{Pitbook} for background).
\begin{proposition}
Let $\{m(t)\}_{0 \le t \le 1}$ be a Brownian meander on $[0,1]$.
Assume that $\phi: \mathbb{R} \to \mathbb{R}$ be sufficiently smooth and $\int_0^1\phi(x)^2dx < \infty$.
We have:
\begin{equation}
\frac{1}{2}\int_0^1\left( \frac{\phi(t)}{M(t)} - M(t) \left( \frac{\phi(t)}{t} + \phi'(t)\right)\right)dt + \frac{M(1)\phi(1)}{4}
\stackrel{d}{=} \mathcal{N}\left(0, \frac{1}{4} \int_0^1 \phi(x)^2 dx\right),
\end{equation}
which is independent of $M(1)$.
In particular, for $\alpha > 0$,
\begin{equation}
\frac{1}{2} \int_0^1 \left(\frac{t^\alpha}{M(t)} - (\alpha +1)t^{\alpha - 1} M(t)\right)dt + \frac{M(1)}{4}\stackrel{d}{=} \mathcal{N}\left(0, \frac{1}{4(1+ 2 \alpha)}\right).
\end{equation}
\end{proposition}
\begin{proof}
For $\{Y_t\}_{0 \le t \le 1}$ a process with local times $\{\ell_x\}_{x \ge 0}$,
define
\begin{equation*}
\hat{L}_t(Y):=\ell_{\tau(t)}, \quad \mbox{with } \tau(t):= \sup\left\{y \ge 0: \int_y^\infty \ell_x dx > t\right\}.
\end{equation*}
By \cite[Corollary 16]{Pit99}, if $Y$ is a reflected Brownian motion on $[0,1]$,
then $\hat{L}(Y)/2$ is a Brownian meander. 
It suffices to apply Proposition \ref{prop:RB} to conclude.
\end{proof}

\bigskip
{\bf Acknowlegdment}: 
We thank Misha Shkolnikov for helpful discussions years back.
This research is supported by NSF CAREER Award DMS-2538791, the Tang
Family Assistant Professorship
and a Columbia-CityU/HK collaborative project that is supported by InnoHK Initiative, The Government of the HKSAR and the AIFT Lab.

\bibliography{unique}
\bibliographystyle{abbrv}

\end{document}